\documentclass{amsart}
\usepackage{amssymb}

\newtheorem{theorem}{Theorem}[section]
\newtheorem{lemma}[theorem]{Lemma}

\newtheorem{setup}[theorem]{}

\newtheorem{_definition}[theorem]{Definition}

\newtheorem{_convention}[theorem]{Convention}
\newenvironment{convention}{\begin{_convention}\rm}{\end{_convention}}

\newtheorem{_example}[theorem]{Example}
\newenvironment{example}{\begin{_example}\rm}{\end{_example}}

\newtheorem{_remark}[theorem]{Remark}
\newenvironment{remark}{\begin{_remark}\rm}{\end{_remark}}

\newtheorem{_claim}[theorem]{Claim}
\newenvironment{claim}{\begin{_claim}\rm}{\end{_claim}}
\numberwithin{equation}{section}
\numberwithin{table}{section}
\numberwithin{figure}{section}

\newcommand{\C}{\mathord{\mathbb C}}

\renewcommand{\P}{\mathord{\mathbb P}}
\newcommand{\Q}{\mathord{\mathbb Q}}
\newcommand{\R}{\mathord{\mathbb R}}
\newcommand{\Z}{\mathord{\mathbb Z}}

\newcommand{\OO}{\mathord{\mathcal O}}

\newcommand{\alb}{{\text{\rm alb}}}
\newcommand{\Alb}{{\text{\rm Alb}}}

\newcommand{\Aut}{\text{\rm Aut}}

\newcommand{\diag}{{\text{\rm diag}}}

\newcommand{\Gal}{\text{\rm{Gal}}}

\newcommand{\id}{\text{\rm id}}

\newcommand{\Ker}{\text{\rm Ker}}
\newcommand{\MRC}{{\text{\rm MRC}}}

\newcommand{\ord}{\text{\rm ord}}

\newcommand{\Pic}{\text{\rm Pic}}

\newcommand{\rank}{\text{\rm rank}}

\newcommand{\ratmap}
{{\,\cdot\negmedspace\cdot\negmedspace\cdot\negmedspace\to\,}}

\newcommand{\BCC}{\mathbb{C}}

\newcommand{\alg}{\mathrm{alg}}

\begin{document}
\begin{large}
\title
[Endomorphisms of Complex Manifolds]
{Cohomologically hyperbolic endomorphisms of complex manifolds}
\author{De-Qi Zhang}
\address
{%
\textsc{Department of Mathematics} \endgraf
\textsc{National University of Singapore, 2 Science Drive 2,
Singapore 117543, Singapore
\endgraf
and
\endgraf
Max-Planck-Institut f\"ur Mathematik,
Vivatsgasse 7,
53111 Bonn, Germany
}}
\email{matzdq@nus.edu.sg}



\thanks{}

\maketitle
%
%
%
%
%
%
%
\begin{abstract}
We show that if a compact K\"ahler manifold $X$
admits a cohomologically hyperbolic surjective endomorphism
then its Kodaira dimension is non-positive. This gives an affirmative answer
to a conjecture of Guedj in the holomorphic case.
The main part of the paper is to determine the geometric
structure and the fundamental groups (up to finite index) for those $X$ of dimension $3$.
\end{abstract}
\section{Introduction}
We work over the field $\C$ of complex numbers.
\par
Let $X$ be a compact K\"ahler manifold of dimension $n \ge 2$.
For a surjective endomorphism $f : X \to X$, one
can define the $i$-th dynamical degree as
$$d_i(f) := \lim_{s \to \infty}
\root{s}\of{\int_X (f^s)^* \omega^i \wedge \omega^{n - i}}$$
where $\omega$ is any K\"ahler form of $X$; see the remarks before \cite[Theorem 1.4]{DS}.
Similarly, one can define $d_i(f)$ for a dominant meromorphic map $f$,
which are bimeromorphic invariants
(independent of the choice of the bimeromorphic model $X$); see \cite[Introduction]{DS}.
It is known that $d_{\ell}/d_{\ell + 1}(f)$ is a non-decreasing function in $\ell$
for $0 \le \ell \le n - 1$; this is called
the Khovanskii - Teissier inequality, initially proved for projective
manifolds, where the general K\"ahler case
was done in \cite{Gr}.

$f$ is said to be {\it cohomologically hyperbolic}
in the sense of \cite{Gu06},
if there is an $\ell \in \{1, 2,  \dots, n\}$
such that the $\ell$-th dynamical degree
$$d_{\ell}(f) > d_i(f) \hskip 2pc
{\rm for \hskip 1pc all}  \hskip 1pc (\ell \ne) \,\, i \,\, \in \,\, \{0, 1, \dots, n\}$$
(or equivalently, for both $i = \ell \pm 1$, by
the Khovanskii - Teissier inequality above). Here
we set $d_0(f) = 1$ and $d_{n+1}(f) = 0$.
\par
In his papers \cite{Gu05} - \cite{Gu06},
Guedj assumed that a dominant rational self map $f : X \ratmap X$
has {\it large topological degree} (i.e., it is cohomologically hyperbolic with $\ell = \dim X$
in the definition above), and constructed a canonical
$f_*$-invariant measure $\mu_f$
which is ergodic with strictly positive Lyapunov
exponents and which can be approximated by repulsive periodic points.
Further,
using the main result of \cite{DS},
the measure is proved to be of maximal entropy and ergodic
for $f$-periodic repulsive points, and with strictly
positive Lyapunov exponents.
As pointed out by the referee,
in general, so far, no good upper bound has been proved for the number of
repulsive periodic points for a general $X$
(but see \cite[Theorem 3.1]{Gu05}),
so one can not say yet that the repulsive periodic points are
equidistributed; for the approximation, one has to choose some of them.

In \cite{Gu06}, Guedj classified cohomologically
hyperbolic rational self maps of surfaces $S$ and deduced that the Kodaira dimension
$\kappa(S) \le 0$. Then he conjectured that the same should hold in higher dimension.
\par
The main part of this paper is to
determine the geometric structure for projective threefolds
having a cohomologically hyperbolic surjective and \'etale endormorphism.
The result below is part of the more detailed one in Theorem \ref{ThB}.
\begin{theorem}\label{ThB1}
Let $V$ be a smooth projective threefold and let $f \colon V \to V$ be a
surjective and cohomologically hyperbolic \'etale endomorphism.
Then one of the following cases occurs;
see \S $\ref{ex}$ for some realizations.
\begin{itemize}
\item[(1)]
$V$ is $f$-equivariantly birational to a $Q$-torus
in the sense of \cite{Ny}.
\item[(2)]
$V$ is birational to a weak Calabi-Yau variety, and $f \in \Aut(V)$.
\item[(3)]
$V$ is rationally connected in the sense of \cite{Cp} and \cite{KoMM},
and $f \in \Aut(V)$.
\item[(4)]
The albanese map $V \to \Alb(V)$ is a smooth and surjective morphism onto
the elliptic curve $\Alb(V)$ with every fibre
a smooth projective rational surface of Picard number $\ge 11$.
Further, the dynamical degrees
satisfy $d_2(f) > d_1(f) \ge \deg(f) \ge 2$.
\item[(5)]
$V$ is $f$-equivariantly birational to the quotient space
of a product ${\rm (Elliptic \ curve)} \times$ $(K3)$ by a finite
and free action. Further, the dynamical degrees satisfy
$d_2(f) > d_1(f) \ge \deg(f) \ge 2$.
\end{itemize}
\end{theorem}
\par
The result below gives an affirmative answer
to the above-mentioned conjecture of Guedj \cite{Gu06} page 7
for holomorphic endomorphisms (see \cite[Theorem 1.3]{Zh2} for the case
of automorphisms on threefolds).
The proof is given very simply by making use of results in \cite{NZ}.
It is classification-free and for arbitrary dimension.
\begin{theorem}\label{ThA}
Let $X$ be a compact complex K\"ahler manifold and
$f : X \to X$ a surjective and cohomologically hyperbolic endomorphism.
Then the Kodaira dimension
$\kappa(X) \le 0$.
\end{theorem}
\par
We can also determine the topological fundamental groups (up to finite index) for those
threefolds admitting a cohomologically hyperbolic \'etale endomorphism.
\begin{theorem}\label{top}
Let $X$ be a smooth projective threefold admitting a
surjective and cohomologically hyperbolic \'etale endomorphism $f$.
Then either $\pi_1(X)$ is finite, or
$\pi_1(X)$ contains a finite-index subgroup isomorphic to either one of:
$$\Z^{\oplus 2}, \hskip 1pc \Z^{\oplus 6}.$$
\end{theorem}
%
%
%
%
Our approach is algebro-geometric in nature; see \cite{Fm},
\cite{FN}, \cite{Og06}, \cite{Og08}, \cite{KOZ},
\cite{NZ}, \cite{Zh1}, \cite{Zh2} and \cite{ICCM}
for  similar approach.
\begin{convention}
We shall use the conventions of Hartshorne's book, \cite{KMM} and \cite{KM}.
\begin{itemize}
\item[(1)]
A normal projective variety $X$ is {\it minimal} if it is $\Q$-factorial, has at worst
terminal singularities and the canonical divisor $K_X$ is nef.
\item[(2)]
A minimal projective variety $X$ is a {\it weak Calabi-Yau variety} if
$K_X \sim_{\Q} 0$ and $q^{\max}(X) = 0$.
Here
$$q^{\max}(Z) := \max \{q(Y) \ | \ Y \to Z \,\,\, \text{\rm finite \'etale} \}.$$
\par \noindent
A minimal projective variety $X$ of dimension $n$ is a {\it Calabi-Yau variety}
if
$$K_X \sim 0, \hskip 1pc \pi_1(X) = (1), \hskip 1pc H^i(X, \OO_X) = 0 \hskip 1pc (1 \le i \le n - 1).$$
\item[(3)]
A morphism $\sigma \colon X \to Y$ is $f$-{\it equivariant}
if there are endomorphisms $f = f|X \colon X \to X$ and $f = f|Y \colon Y \to Y$
such that
$\sigma \circ (f|X) = (f|Y) \circ \sigma$.
\item[(4)]
In this paper, every endomorphism on a compact K\"ahler manifold (or a projective variety)
is assumed to be surjective, so it is also finite, and even \'etale
when the Kodaira dimension $\kappa(X) \ge 0$; see \cite[Lemma 2.3]{Fm}.
\end{itemize}
\end{convention}
\subsection*{Acknowledgement}
The author would like to thank Noboru Nakayama for very carefully reading,
critical comments and valuable suggestions.
He also likes to thank the referee for constructive suggestions
and useful references.
The author is grateful to the Max-Planck-Institute for Mathematics at Bonn
for the warm hospitality in the first quarter of the year 2007.
This project is supported by an Academic Research Fund of NUS.
\section{Proofs of Theorems \ref{ThA} - \ref{top}}
In this section, we shall prove Theorems \ref{ThA} and \ref{top}, and also Theorem \ref{ThB}
below which implies Theorem \ref{ThB1} and determines the geometric structure for projective threefolds in
Theorem \ref{ThA}.
\begin{theorem}\label{ThB}
Let $V$ be a smooth projective threefold and let $f \colon V \to V$ be a
surjective and cohomologically hyperbolic \'etale endomorphism.
Then one of the following cases occurs;
see \S $\ref{ex}$ for some realizations.
\begin{itemize}
\item[(1)]
$\kappa(V) = 0$ and $q^{\max}(V) = \dim V = 3$. Further,
$V$ is $f$-equivariantly birational to a $Q$-torus
in the sense of \cite{Ny}.
To be precise, there are an $f$-equivariant
birational morphism $V \to X$ and an $f$-equivariant \'etale Galois
cover $Y \to X$ from an abelian variety $Y$.
\item[(2)]
$\kappa(V) = 0 = q^{\max}(V)$, $\pi_1(V)$ is finite, and $f \in \Aut(V)$.
Further, $V$ is birational to a weak Calabi-Yau variety.
\item[(3)]
$\kappa(V) = -\infty$, $q^{\max}(V) = 0$, $\pi_1(V) = (1)$ and $f \in \Aut(V)$.
Further, $V$ is rationally connected in the sense of \cite{Cp} and \cite{KoMM}.
\item[(4)]
$\kappa(V) = -\infty$, $q^{\max}(V) = q(V) = 1$ and the dynamical degrees
satisfy $d_2(f) > d_1(f) \ge \deg(f) \ge 2$.
Further, the albanese map $V \to \Alb(V)$ is smooth and surjective with every fibre
$F$ a smooth projective rational surface of Picard number $\ge 11$.
\item[(5)]
$\kappa(V) = 0$, $q^{\max}(V) = 1$ and the dynamical degrees satisfy
$d_2(f) > d_1(f) \ge \deg(f) \ge 2$.
Further, $V$ is $f$-equivariantly birational to the quotient space
of a product ${\rm (Elliptic \ curve)} \times$ $(K3)$
by a finite and free action.
To be precise, $V$ has a unique minimal model
$X$ and $f | V$ induces a finite \'etale endomorphism $f | X$ on $X$.
There is an $f$-equivariant \'etale Galois cover $Y = E \times S \to X$
with $E$ an elliptic curve and $S$ a (smooth and minimal) $K3$ surface,
such that $f | Y = (f | E) \times (f | S)$ for some isogeny $f | E$ with $\deg(f|E) = \deg(f| V)$
and $f | S \in \Aut(S)$ of positive entropy.
\end{itemize}
\end{theorem}
\begin{remark}
$ $
\begin{itemize}
\item[(1)]
See \cite{Fm} and \cite{FN} for the case
where $\kappa(V) \ge 0$ and $\deg(f) \ge 2$.
\item[(2)]
The \'etaleness of $f$ in Theorems \ref{top} and \ref{ThB} above
is automatic if either $\deg(f) = 1$, or if $\deg(f) \ge 2$
and $\kappa(X) \ge 0$; see \cite[Lemma 2.3]{Fm}.
\item[(3)]
In Theorem $\ref{ThB}$ $(1)$ and $(5)$, we have
$d_i(f|V) = d_i(f|Y)$ for all $i$.
In general,
we have $d_j(g|V) = d_j(g|W)$ for all $j$
if $V \to W$ is a $g$-equivariant generically finite morphism;
see \cite[Appendix, Lemma A.8]{NZ}.
\end{itemize}
\end{remark}
\begin{setup} {\bf Proof of Theorem \ref{ThA}.}
\end{setup}
We will make use of \cite[Theorem A, and Appendix]{NZ}.
Suppose the contrary that the Kodaira dimension $\kappa(X) \ge 1$.
Then $f \colon X \to X$ is a finite \'etale morphism (see \cite[Lemma 2.3]{Fm}).
We choose $m \gg 0$ such that
$$\Phi_m = \Phi_{|mK_X|} \colon X \ratmap W_m \subseteq \P(H^0(X, mK_X))$$
gives rise to the Iitaka fibring. By \cite[Theorem A]{NZ}, $f$ induces
an automorphism $f_m \colon W_m \to W_m$ of finite order, such that
$\Phi_m \circ f = f_m \circ \Phi_m$.
Replacing $X$ by an $f$-equivariant resolution of base locus of $|mK_X|$ due to Hironaka
(see also \cite[\S 1.4]{NZ}), we may assume that $\Phi_m$ is a well defined morphism.
Now the theorem follows from the result below, noting that $\dim W_m = \kappa(X) \ge 1$.
\begin{lemma}\label{periodic}
Let $\pi\colon X \to Y$ be a proper holomorphic map from a compact K\"ahler manifold $X$ to a compact
complex analytic variety $Y$ with general fibres connected, and let $f \colon X \to X$ and
$f_Y \colon Y \to Y$ be surjective endomorphisms such that $\pi \circ f = f_Y \circ \pi$.
Suppose that $f$ is \'etale, $f_Y$ is an automorphism of finite order and
$f$ is cohomologically hyperbolic. Then $\dim Y = 0$ (and $Y$ is a single point).
\end{lemma}
\begin{proof}
Replacing $f$ by its power, we may assume that $f_Y = \id$.
Let $F$ be a smooth general fibre of $\pi$.
We claim that $f | F$ is also cohomologically hyperbolic. Indeed, by the fundamental work of
Gromov and Yomdin, the topological entropy $h(g)$ of an endomorphism $g$ of a compact
K\"ahler manifold
is the maximum of logarithms $\log d_i(g)$ of dynamical degrees.
So suppose that for some $1 \le r \le k := \dim F$, we have:
$$h(f|F) = \log d_r(f|F) = \max_{1 \le i \le k} d_i(f|F).$$
By \cite[Appendix, Proposition A.9 and Theorem D]{NZ},
we have:
$$h(f|F) = \log d_r(f|F) \le \log d_r(f|X) \le h(f | X) = h(f|F),$$
so $d_r(f|F) = d_r(f|X)$. Now for any $i \ne r$, by
[ibid.], we have:
$$d_i(f|F) \le d_i(f|X) < d_r(f|X) = d_r(f|F).$$
Here the strict inequality holds because $f|X$ is cohomologically hyperbolic.
This proves the claim.
\par
On the other hand, note that $\deg(f|X) = \deg(f|F)$. Hence, by [ibid.],
we have the following, with $n = \dim X$ and $k = \dim F$:
$$h(f|F) = \log d_r(f|F) \le \log d_{r + n - k}(f | X) \le \log d_r(f|X) = h(f|F).$$
Thus all inequalities above become equalities; since
$f|X$ is cohomologically hyperbolic, the maximality of $d_r$ implies that $r + n - k = r$.
So $n = k$ and $Y$ is a point. This proves the lemma and also Theorem \ref{ThA}.
\end{proof}
We need the result below in the proof of Theorem \ref{ThB}.
\begin{lemma}\label{d-compare}
Let $X$ be a compact K\"ahler manifold of dimension $n$ and let $f\colon X \to X$ be a surjective
endomorphism. Then the dynamical degrees satisfy
$d_{n-i}(f) = d_{i}(f^{-1}) (\deg(f))$. Here
$d_j(f^{-1})$ denotes the spectral radius of the linear transformation
$$(f^*)^{-1} \colon H^{j,j}(X, \C) \to H^{j,j}(X, \C).$$
\end{lemma}
\begin{proof}
One can use the fact that $f_* f^* = (\deg(f)) \id$ on the cohomology ring of $X$
to give a simple proof. Below is another elementary proof.
Set $s = h^{i,i}(X, \C) = h^{n-i,n-i}(X, \C)$.
Let $\{e_1, \dots, e_s\}$ and $\{\varepsilon_1, \dots, \varepsilon_s\}$ be dual bases of
$H^{i,i}(X, \C)$ and $H^{n-i,n-i}(X, \C)$ with respect to the perfect pairing
below such that $e_i . \varepsilon_j = \delta_{ij}$ (Kronecker's symbol):
$$H^{i,i}(X, \C) \times H^{n-i,n-i}(X, \C) \to \C.$$
Let $A$ (resp. $B$) be the matrix representation of
$(f^*)^{-1} | H^{i,i}(X, \C)$ (resp. $f^* | H^{n-i, n-i}(X, \C)$).
Then a calulation in linear algebra implies that
$B = (\deg(f)) A^T$. The lemma follows.
\end{proof}
\begin{setup} {\bf Proof of Theorem \ref{ThB}.}
\end{setup}
By Theorem \ref{ThA}, $\kappa(V) \le 0$.
Our Theorem \ref{ThB} follows from the three lemmas below.
\begin{lemma}
Theorem $\ref{ThB}$ is true when $\kappa(V) = 0$.
\end{lemma}
\begin{proof}
We will make use of \cite[Theorem B]{NZ}.
Let $f\colon V \to V$ be as in the theorem.
Let $X$ be a ($\Q$-factorial) minimal model of $V$ with at worst terminal
singularities, whence $K_X \sim_{\Q} 0$ (see \cite{Mi}, \cite{Ka}).
Then $f|V$ induces a dominant
rational map $f \colon X \ratmap X$, which is nearly \'etale in the sense of
\cite[\S 3]{NZ}. By \cite[Theorem B and its Remark]{NZ},
either an \'etale cover $\widetilde{X}$ of $X$ is a weak Calabi-Yau variety, or
there are an \'etale cover $\tau \colon F \times A \to X$,
an automorphism $\varphi_F \colon F \to F$
and a finite \'etale endomorphism $\varphi_A \colon A \to A$ with $\deg(\varphi_A) = \deg(f)$
such that $f \circ \tau = \tau \circ (\varphi_F \times \varphi_A)$.
Here $F$ is either a point (and hence $A$ is a 3-torus), or K3 or Enriques (and $A$ is an elliptic curve),
by the classification of lower dimensional weak Calabi-Yau varieties.
\par
By further \'etale cover (to the Galois closure), we may assume that $\tau$ is Galois.
Replacing $\tau$, we may also reduce the Enriques case of $F$ to the
$K3$ case.
\par
The case above involving $\widetilde{X}$ fits Theorem \ref{ThB} (2).
Indeed, $\pi_1(V) = \pi_1(X)$ (see \cite[Theorem 7.8]{Ko} and \cite[Theorem 1.1]{Ty}), so
$\pi_1(V)$ is finite by \cite[Corollary 1.4]{NS}.
\par
Consider the case $\dim A = 3$. Then $X$ is a $Q$-torus in the sense of
\cite{Ny}. Both the birational map $V \ratmap X$ and the dominant rational map $f \colon X \ratmap X$ are
well defined morphisms by the absence of rational curves on tori and Hironaka's resolution
of indeterminancy of a rational map; see also \cite[Lemma 9.11]{Un}.
So Theorem \ref{ThB}(1) occurs.
\par
Consider the case $\dim A = 1$. We shall show that
Theorem \ref{ThB} (5) takes place. Note that $F \times A$ is the unique minimal model
of its biraitonal class, up to isomorphism. This is because other minimal models are obtained from $F \times A$
by a finite sequence of flops with centre a union of rational curves which must be
contained in some fibres of $F \times A \to \Alb(F \times A) = A$, i.e., contained in the
$K3$ surfaces $F$. However, we assert that $F \times A$ admits no flop.
Indeed, such a flop induces a non-isomorphic
birational automorphism of $F$, which is an isomorphism away from a few rational curves,
and hence is indeed an isomorphism by the uniqueness of a surface minimal model, absurd!
So the assertion is true. This assertion also appeared in
\cite[page 66]{Fm}.
\par
Next we claim that $X$ is the unique minimal model in its birational class, up to isomorphism.
This claim appeared in \cite[page 61]{Fm}. We prove it for
the convenience of the readers.
It is enough to show the assertion of the absence of flops from $X$.
Suppose the contrary that
$\sigma\colon X \ratmap X'$ is a flop to another minimal model.
Then $X'$ is also smooth. Since the fundamental group of a smooth variety will not
be changed after a smooth blowup or blowdown and after removing some codimension 2 subsets,
the existence of an \'etale Galois cover $\tau\colon F \times A \to X$ induces an \'etale Galois cover
$\tau' \colon \widetilde{X}' \to X'$ and a birational map
$\tilde{\sigma}\colon F \times A \ratmap \widetilde{X}'$ lifting the flop $\sigma\colon X \ratmap X'$
and being isomorphic in codimension one. So $\tilde{\sigma}$ is either an isomorphism
or a composition of flops. The absence of such flop as shown in the paragraph above,
implies that $\tilde{\sigma}$ is indeed an isomorphism. The consideration of
the fundamental group again implies that $\Gal((F \times A)/X)$ and
$\Gal(\widetilde{X}'/X')$ are conjugate to each other whose quotients are hence isomorphic via the initial map $\sigma$.
But $\sigma$, being a flop, is not isomorphic. We reach a contradiction. Hence both the assertion and
the claim are true.
\par
Applying \cite[Lemma 3.2]{NZ} to $f \colon X \ratmap X$, we see that
$f$ is the composition of a birational map $\gamma\colon X \ratmap X''$ and a finite \'etale
morphism $X'' \to X$. Thus $X''$ is also a minimal model and hence $\gamma$
is either an isomorphism or a composition of flops. The assertion in the paragraph above
implies that $\gamma$ is an isomorphism. So our initial $f$ is indeed a well defined
finite \'etale morphism.
Thus Theorem \ref{ThB} (5) takes place,
where we set $S:= F$, $E:= A$ and $Y:= F \times A$.
Indeed, since $d_i(f|V) = d_i(f|Y)$ by \cite[Appendix, Lemma A.8]{NZ}
and applying the K\"unneth formula, we have:
$$d_i(f|V) = \max_{0 \le s \le i}\{d_{s}(f|S) \ d_{i-s}(f|E)\}.$$
Since $f|V$ is cohomologically hyperbolic, we must have $d_1(f|S) \ge 2$ and $d_1(f|E) \ge 2$,
whence the inequalities about the dynamical degrees follow.
Also, since $\pi_1(Y) = \pi_1(E)$, we see that ($q^{\max}(V) =$) $q^{\max}(Y) = 1$.
This completes the proof of the lemma.
\end{proof}
Next we consider the case
$\kappa(V) = -\infty$. The completed good minimal model program
for threefolds (see \cite{KMM} or \cite{KM}), implies that $V$ is uniruled.
Let $\MRC_{V}\colon V \ratmap Y'$ be a maximal rationally connected fibration
in the sense of \cite{Cp} and \cite{KoMM}. Then $Y'$ is not uniruled
by \cite[(1.4)]{GHS}. So $\kappa(Y') \ge 0$.
By \cite[Theorem C and its Remark]{NZ}, there are a birational morphism
$X \to V$ from a smooth projective variety, and a smooth projective variety
$Y$ birational to $Y'$, such that $f|V$ induces a finite \'etale endomorphism
$f \colon X \to X$, the induced maps $\pi := \MRC_X\colon X \to Y$ and
$f_Y \colon Y \to Y$ are well defined morphisms,
and $\pi \circ f = f_Y \circ \pi$. Further, $\deg(f) = \deg(f_Y)$.
\par
Since a torus contains no rational curves, we have $\Alb(V) = \Alb(X) = \Alb(Y)$.
Further, the composition $V \ratmap X \to Y \to \Alb(Y)$ is the well defined
albanese morphism $\alb_{V}$. Note also that $\kappa(Y) = \kappa(Y') \ge 0$
and hence $f_Y$ is finite \'etale.
\par
If $\dim Y = 0$, then Theorem \ref{ThB} (3) takes place
because a rationally connected smooth projective variety is simply-connected (see \cite{Cp}),
whence $\deg(f) = 1$.
We now consider the cases $\dim Y = 1, 2$ separately.
\begin{lemma}
Assume that $\kappa(V) = - \infty$ and $\MRC_{V}(V)$ is a curve.
Then Theorem $\ref{ThB} (4)$ takes place. Further, for the $F$ there, the anti-canonical
divisor $-K_F$ is not big and $K_F^2 < 0$.
\end{lemma}
We now prove the lemma. By Lemma \ref{periodic}, $f_Y$ is not periodic.
So $Y$ is an elliptic curve, noting that $\kappa(Y) \ge 0$. Further, either $f_Y$ is an isogeny
of $\deg(f_Y) \ge 2$, or $f_Y$ is a translation of infinite order
and hence $\deg(f) = \deg(f_Y) = 1$ (so both $f$ and $f_Y$ are automorphisms).
\par
We claim that $\pi \colon X \to Y$ is a smooth morphism.
Indeed, suppose the contrary that we have a non-empty set $D(X/Y)$, the discriminant locus of $\pi$,
i.e., the subset of $Y$ over which $\pi \colon X \to Y$ is not smooth.
Since $f \colon X \to X$ is \'etale and is the lifting of $f_Y$, we have
$f_Y^{-1}(D(X/Y)) = D(X/Y)$. Replacing $f$ by its power, we may assume that $f_Y$ fixes
every point in $D(X/Y)$. This contradicts the description of $f_Y$ above. Therefore,
$\pi\colon X \to Y$ is smooth
and every fibre of it is a {\it smooth} rational projective surface.
\par
Note that $\pi = \alb_X$. By the same reason, $\alb_{V} \colon V \to \Alb(V) = Y$
is smooth.
For the clean-ness of the notation, we replace
$(V, f)$ by $(X, f)$.
\par
When $f$ is an automorphism, we have $d_1(f) = d_2(f)$ by \cite[Lemma 2.8]{Zh2}.
This contradicts the concavity as mentioned in Claim \ref{claim:noKE} below.
Therefore, $f_Y$ is an isogeny with $\deg(f) = \deg(f_Y) \ge 2$.
\par
Let $0 \ne v_{f^{\pm}}$ be nef $\R$-divisors such that
$$(f^{\pm})^* v_{f^{\pm}} = d_1(f^{\pm}) v_{f^{\pm}},$$
guaranteed by a result of Birkhoff \cite{Bi} generalizing the Perron-Frobenius
theorem to (the nef) cone. Let $F$ be a fibre of $\pi  = \alb_X \colon X \to Y$.
So $F$ is a smooth projective rational surface.
\begin{claim}\label{claimB1}
The following are true.
\begin{itemize}
\item[(1)]
$f^*F$ is a disjoint union of $\deg(f)$ fibres;
\par \noindent
$f^*F \equiv \deg(f) F$; $d_i(f) \ge \deg(f)$ for both $i = 1, 2$.
\item[(2)]
$0 = v_f \cdot F \cdot K_X = (v_f)|F \cdot K_F$.
\item[(3)]
$d_2(f) = d_1(f^{-1}) \deg(f) \ge \deg(f)$ and $d_1(f^{-1}) \ge 1$.
\item[(4)]
If $v_f \cdot F = 0$ then $v_f \equiv e F$ for some $e > 0$ and $d_2(f) > d_1(f) = \deg(f)$.
\item[(5)]
If $v_f \cdot F \ne 0$ then $d_2(f) \ge d_1(f) \deg(f) > d_1(f) \ge \deg(f)$.
\end{itemize}
\end{claim}
\begin{proof}
(1) The first two assertions are true because $f$ is \'etale and
$F$, being rational, is simply connected. In particular, $d_1(f) \ge \deg(f)$.
Applying $f^*$ to the non-zero cycle $K_X \cdot F = K_F$, we get $d_2(f) \ge \deg(f)$.
\par
(2) If $c:= v_f \cdot F \cdot K_X \ne 0$, then
$\deg(f)  c = f^* c = d_1(f)  \deg(f) c$ and $d_1(f) = 1$. This is absurd
because $f$ is of positive entropy.
\par
(3) follows from (2) and Lemma \ref{d-compare}.
\par
(4) The first part follows from the Lefschetz hyperplane section theorem to reduce
to the Hodge index theorem for surfaces (see the proof of \cite[Lemma 2.6]{Zh2}),
while the second follows from the first by applying $f^*$,
the assertion (1), and $f$ being cohomologically hyperbolic.
\par
(5) is similar to (4) by applying $f^*$.
\end{proof}
\par
It remains to show the assertion
that $-K_F$ is not big, and $\rank \ \Pic(F) \ge 11$ or equivalently $K_F^2 \le -1$.
Consider the case where $-K_F$ is big or $K_F^2 \ge 0$, and we shall derive a contradiction.
If $K_F^2 \ge 1$, then $-K_F$ is big by the Riemann-Roch theorem
applied to $-nK_F$.
Thus we assume that either $K_F^2 =  0$
or $-K_F$ is big. This assumption and
Claim \ref{claimB1} (2) imply $(v_f) | F \equiv \alpha K_F = \alpha K_X  | F$
for some $\alpha \ne 0$ (by Claim \ref{claimB2} below).
Applying $f^*$, we get $d_1(f) = 1$, absurd. Therefore, the assertion is true. The lemma then follows.
Indeed, $q^{\max}(V) = q^{\max}(Y)$ ($= 1$) because $\pi_1(V) = \pi_1(Y)$ as in the proof of Theorem \ref{ThB}
at the end of this section.
\begin{claim}\label{claimB2}
Suppose $K_F^2 = 0$ or $-K_F$ is big.
Then the cohomology class of $v_f$ is not a multiple of that of $F$,
so $(v_f) \cdot F$ is not homologous to zero.
\end{claim}
\begin{proof}
Suppose the contrary that the claim is false. Applying $f^*$, we get $d_1(f) = \deg(f)$.
Since $f$ is cohomologically hyperbolic and by Claim \ref{claimB1} (1), we have $d_2(f) > \deg(f)$,
and hence $d_1(f^{-1}) > 1$ by Claim \ref{claimB1} (3). The latter and the proof of Claim \ref{claimB1} (2)
imply that $0 = v_{f^{-1}} \cdot F \cdot K_X = (v_{f^{-1}} | F) \cdot K_F$.
Then by the assumption on $-K_F$ and the Hodge index theorem (see \cite[IV, Cor. 7.2]{BHPV}),
we have $v_{f^{-1}} | F \equiv a K_F = a K_X | F$ for some scalar $a$. If $a \ne 0$,
applying $f^*$ to the equality, we get $d_1(f^{-1}) = 1$, absurd.
If $a = 0$, then $v_{f^{-1}} | F \equiv 0$ and hence
$v_{f^{-1}} \equiv b F$ for some $b > 0$ by the Lefschetz hyperplane
section theorem to reduce to the Hodge index theorem for surfaces.
Applying $f^*$, we get $1 > 1/(d_1(f^{-1})) = \deg(f) > 1$, absurd. This proves the claim
and also the lemma.
\end{proof}
\begin{lemma}
In the situation of Theorem $\ref{ThB}$ it is impossible that $\kappa(V) = - \infty$ and $\MRC_{V}(V)$ is a surface.
\end{lemma}
We now prove the lemma.
Consider the case where ($\kappa(X) =$) $\kappa(V) = - \infty$ and $\MRC_{V}(V)$ (or
eqivalently $Y = \pi(X)$) is a surface.
If $\kappa(Y) \ge 1$, then after equivariant modification, we may assume that
for some $n > 0$, the map $\Phi_{|nK_Y|} \colon Y \to Z$ is a well defined morphism
giving rise to the Iitaka fibring. By \cite[Theorem A]{NZ},
$f_Y$ descends to an automorphism $f_Z \colon Z \to Z$ of finite order.
Note that $\dim Z = \kappa(Y) \ge 1$. This contradicts Lemma \ref{periodic}.
\par
Therefore, $\kappa(Y) = 0$. We may assume that $Y$ is minimal. This can be achieved if
$\deg(f_Y) \ (= \deg(f)) \ = 1$ by equivariant blowdown; on the other hand, if $\deg(f_Y) \ge 2$,
then $Y$ has no negative $\P^1$ and hence $Y$ is already minimal, for otherwise, iterating $f^{-1}$
will produce infinitely many disjoint negative $\P^1$ (noting that $\P^1$ is simply connected and $f_Y$ is \'etale),
contradicting the finiteness of the Picard number of $Y$; see \cite[page 43]{Fm}.
Thus, $Y$ is abelian, hyperelliptic, $K3$ or Enriques.
\begin{claim}\label{claim:noKE}
$Y$ is neither $K3$ nor Enriques.
\end{claim}
\begin{proof}
The claim is clear when the \'etale map $f_Y$ has $\deg(f_Y) \ge 2$ (so $|\pi_1^{\alg}(Y)| = \infty$
by iterating $f_Y$), since $|\pi_1(Y)| \le 2$ when $Y$ is $K3$ or Enriques.
Suppose $f_Y$ (and hence $f$) are automorphisms. By Lemma \ref{periodic},
$f_Y$ is not periodic. If $f_Y$ is of positive entropy, then $d_1(f) = d_2(f)$ as proved in
\cite[Claim 2.11(1)]{Zh2}; this is absurd since $f$ is cohomologically hyperbolic
and by the concavity from the Khovanskii-Teissier inequality as in \cite[Proposition 1.2]{Gu05}.
Thus $f_Y$ is parabolic. Then there is an elliptic fibration $Y \to \P^1$
such that $f_Y$ descends to a periodic automorphism on $\P^1$; see \cite[Lemma 2.19]{Zh2}.
However, this contradicts Lemma \ref{periodic}.
\end{proof}
\par
If $Y$ is a hyperelliptic surface, then $Y$ is a quotient of a torus $Z$ by a group
of order $m = 2, 3, 4$, or $6$ (taking $m$ minimal). Our $f_Y \colon Y = Y_1 \to Y = Y_2$
lifts to an endomorphism $f_Z$ of $Z$.
Indeed, $Z \times_{Y_2} Y_1$ is isomorphic to $Z$ (as the minimal torus cover of $Y_1$).
There is a further lifting $\tilde{f} = f \times f_Z$ on $\widetilde{X} := X \times_Y Z$
so that the projection $\widetilde{X} \to Z$ is just $\MRC_{\widetilde X}$.
Note that $\tilde{f}$ is also cohomologically hyperbolic by \cite[Appendix, Lemma A.8]{NZ}.
We will reach a contradiction by the argument below when $Y$ is an abelian surface.
\par
We now consider the case where $Y$ is an abelian surface.
By Lemma \ref{periodic}, $f_Y$ and its equivariant descents are not periodic.
If $f$ is an automorphism and $f_Y$ is of positive entropy, then by \cite[Claim 2.11 (1)]{Zh2}
we have $d_1(f) = d_2(f)$, which is absurd as mentioned in the proof of Claim \ref{claim:noKE}.
If $f$ is an automorphism and $f_Y$ is rigidly parabolic in the sense of \cite[2.1]{Zh2},
then we will get a contradiction as shown in \cite[Claim 3.13]{Zh2}.
\par
Thus, we may assume that ($\deg(f_Y) =$) $\deg(f) \ge 2$.
\begin{claim}\label{claim:smooth-p1}
$\pi \colon X \to Y$ is a smooth morphism, so every fibre is $\P^1$.
\end{claim}
\begin{proof}
Suppose the contrary that the discriminant locus $D:= D(X/Y)$ is not empty.
Since $f$ is \'etale, we have $f_Y^{-1}(D) = D$, whence
$D$ does not contain isolated points and $D$ is a disjoint union of curves $D_i$.
We may assume that $f_Y^{-1}(D_i) = D_i$ for all $i$
after replacing $f$ by its power.
Further, $\deg(f_Y|D_i) = \deg(f_Y) \ge 2$. Thus $D_i$ is not of general type
and hence $\kappa(D_i) = 0$. By \cite[Theorem 10.3]{Un}, every $D_i$ is an elliptic curve
and a subtorus for $i = 1$ (after changing the origin).
$f_Y$ induces an endomorphism $f_Z$ of the elliptic curve $Z:= Y/D_1$
such that $f_Z^{-1}\{d_i\} = \{d_i\}$ for each $d_i$: the image of $D_i$.
Thus $\ord(f_Z) \le 6$. This contradicts Lemma \ref{periodic}.
So the claim is proved.
\end{proof}
Let $0 \ne v_{f^{\pm}}$ be nef $\R$-divisors such that $(f^*)^{\pm} v_{f^{\pm}} = d_1(f^{\pm}) v_{f^{\pm}}$.
Let $F$ be a fibre of $\pi\colon X \to Y$. So $F \cong \P^1$ by the claim above.
\begin{claim}\label{claim:pullback}
The following are true.
\begin{itemize}
\item[(1)]
$f^*F$ is a disjoint union of $\deg(f)$ fibres; $f^*F \equiv \deg(f) F$.
\item[(2)]
$d_1(f^{-1}) \deg(f) = d_2(f) \ge \deg(f)$ and $d_1(f^{-1}) \ge 1$.
\item[(3)]
$F \cdot v_f = 0$; $v_f \equiv \pi^*H^+$ for some nef divisor $0 \ne H^+$ on $Y$.
\item[(4)]
$F \cdot v_{f^{-1}} = 0$; $v_{f^{-1}} \equiv \pi^*H^-$ for some nef divisor $0 \ne H^-$ on $Y$.
\item[(5)]
$H^+ \cdot H^- \ne 0$.
\end{itemize}
\end{claim}
\begin{proof}
(1) and (2) are as in a claim of the previous lemma.
\par
(3) If $\alpha := F \cdot v_f \ne 0$, then we get a contradiction $d_1(f) = 1$, by the calculation:
$$\deg(f) \alpha = f^* \alpha = (\deg f) d_1(f) \alpha.$$
Hence $F \cdot v_f = 0$. This and $-K_X$ being $\pi$-ample from Claim \ref{claim:smooth-p1},
imply that $v_f \equiv \pi^*H^+$ (see \cite[Lemma 3-2-5]{KMM} or \cite[page 46]{KM}).
Here $H^+$ is nef because so is $v_f$.
\par
(4)  If $\beta := F \cdot v_{f^{-1}} \ne 0$, then we get $d_1(f^{-1}) = 1$
by the calculation:
$$\deg(f) \beta = f^* \beta = \beta \deg(f)/d_1(f^{-1}).$$
Thus $d_2(f) = \deg(f)$ by (2). This contradicts \cite[(1.2)]{Gu05} as argued in
Claim \ref{claim:noKE}.
The rest of (4) is as in (3).
\par
(5) If (5) is false, then $H^+ \equiv \gamma H^-$ for some $\gamma > 0$ by the Hodge index theorem.
Applying $f^* \pi^*$, we get $1 < d_1(f) = 1/d_1(f^{-1}) \le 1$ by (2), absurd!
\end{proof}
By the claim above, we can write $v_f \cdot v_{f^{-1}} \equiv \delta F$ for some $\delta > 0$.
Applying $f^*$, we have
$$d_1(f)/d_1(f^{-1}) = \deg(f) = d_2(f)/d_1(f^{-1}),$$
by Claim \ref{claim:pullback}.
Hence $d_1(f) = d_2(f) \ge \deg(f)$, by Claim \ref{claim:pullback}. This is impossible
because $f$ is cohomologically hyperbolic.
This proves the lemma.
The proof of Theorem \ref{ThB} is also completed.
\begin{setup}
{\bf Proof of Theorem \ref{top}.}
\end{setup}
In view of Theorem \ref{ThB}, we have only to consider the case in
Theorem \ref{ThB}(4). Since a general fibre (indeed every fibre) $F$
of $\alb_{X} \colon X \to \Alb(X)$ is a smooth projective rational surface,
we have $\pi_1(X) = \pi_1(\Alb(X)) = \Z^{\oplus 2}$ (see \cite{Cp}).
This proves the theorem.
\section{Examples}\label{ex}
In this section we give examples to realize some cases in Theorem \ref{ThB}.
\begin{example} Examples for Theorem \ref{ThB} (4)-(5).
\par
Let $Z$ be a compact complex K\"ahler surface with an automorphism $f_Z$ of positive entropy.
Let $E$ be an elliptic curve and $f_E \colon E \to E$ an isogeny of $\deg(f_E) \ge 2$. Set $X := Z \times E$
and $f := f_Z \times f_E$. Then $d_2(f) > d_1(f) \ge d_3(f)$, because
we have the 'product formula':
$$d_i(f) = \max_{0 \le s \le i} \{d_s(f_Z) \ d_{i-s}(f_E)\}$$
by the K\"unneth formula for
cohomologies and \cite[Proposition 5.7]{Di};
alternatively, as pointed out by the referee, to deduce the displayed equality,
one can calclulate the dynamical degrees as in \cite[Introduction]{DS}
in terms of the growth of the pullbacked K\"ahler form of $X$ induced from
those of $Z$ and $E$.
If we take $Z$ to be K3 or Enriques (resp. rational surface)
then $(X, f)$ fits Theorem \ref{ThB} (5) (resp. (4)).
\par
For examples of such $(Z, f_Z)$ of positive entropy, see \cite{Ct}, \cite{McP2}.
\end{example}
\begin{example} Cohomologically hyperbolic endomorphisms on rational varieties.
\par
Let $S_i$ ($1 \le i \le r$) be a smooth projective rational surface and $f_i$ an automorphism of $S_i$
of positive entropy; see \cite{McP2} for such examples.
Set
$$X := S_1 \times \cdots \times S_r, \hskip 2pc f := f_1 \times \cdots \times f_r \ \in \ \Aut(X).$$
Then using the 'product formula' as in the previous example, we see that
$f$ is cohomologically hyperbolic with $d_r(f) > d_i(f)$ for all $i \ne r$.
\par
Let $f_p \colon \P^1 \to \P^1$ be an endomorphism of $\deg(f_p) \ge 2$.
Set
$$Y := \P^1 \times S_1 \times \cdots \times S_r, \hskip 2pc
f_Y := f_p \times f_1 \times \cdots \times f_r.$$
Then using the 'product formula' again, we see that
$f_Y$ is a cohomologically hyperbolic endomorphism with
$d_{r+1}(f_Y) > d_i(f_Y)$ for all $i \ne r+1$.
However, $f_Y$ is not \'etale because $\P^1$ is simply connected and hence
$f_p$ is not \'etale.
\end{example}
\begin{example} Cohomologically hyperbolic rational self maps on smooth Calabi-Yau.
\par
Denote by $\zeta_s = \exp(2 \pi \sqrt{-1}/s)$, a primitive $s$-th root of $1$.
Let $E = \C/(\Z + \Z \zeta_3)$ be an elliptic curve admitting
a group automorphism $f_E$ of order 3. Set $A_3 = E \times E \times E$
and $f_3 = \diag[f_E, f_E, f_E]$. Then $f_3$ acts on $A$ with $27$ fixed points.
\par
Consider the Klein quartic curve below
$$C := \{X_0X_1^3 + X_1X_2^3 + X_2X_0^3 = 0\} \subset \P^2.$$
which is of genus $3$ and
with $|\Aut(C)| = 42 \deg(K_C)$ (reaching the Hurwitz upper bound).
Indeed, $\Aut(C) = L_2(7)$, a simple group of order $168$.
Let
$$f_C \colon [X_0\colon X_1 \colon X_2] \mapsto [\zeta_7X_0 \colon \zeta_7^2X_1 \colon \zeta_7^4X_2].$$
be an order-7 automorphism of $C$. Let $A_7 = J(C)$ be the Jacobian abelian threefold
and let $f_7 = \diag[\zeta_7, \zeta_7^2, \zeta_7^4]$ be the induced order-7 automorphism on $A_7$.
\par
For $A_n$ ($n = 3, 7$), let $\overline{X}_n = A_n/\langle f_n \rangle$.
Thanks to the work of Oguiso-Sakurai \cite[Theorem 3.4]{OS},
there is a crepant desingularization $X_n \to \overline{X}_n$,
and $X_n$ satisfies the following:
$$K_{X_n} \sim 0, \hskip 2pc \pi_1(X_n) = (1).$$
Note that $K_{{\overline X}_n} \sim 0$.
By \cite[Theorem 7.8]{Ko}, $\pi_1(\overline{X}_n) = \pi_1(X_n) = (1)$.
Thus by the Serre duality, $X_n$ is a smooth Calabi-Yau variety, while $\overline{X}_n$ is
a Calabi-Yau variety but with isolated {\it canonical} singularities.
For $m \ge 2$, let $m_n \colon A_n \to A_n$, $a \mapsto m . a$, be an endomorphism
of degree $m^6$. Then $\Ker(m_{A_n}) = (\Z/(m))^{\oplus 6}$. The group below
of order $n . m^6$ acts on $A_n$ faithfully
$$G_m := ((\Z/(m))^{\oplus 6}) \rtimes \langle f_n \rangle.$$
\par
$m_n$ induces an endomorphism $\overline{m}_n \colon \overline{X}_n \to \overline{X}_n$ of
degree $m^6$. Note that $m_n$ is cohomologically hyperbolic,
and hence so is $\overline{m}_n$ by \cite[Appendix, Lemma A.8]{NZ}.
The pairs $(\overline{X}_n, \overline{m}_n)$ with $n = 3, 7$ and $m \ge 2$,
are close to the situation in Theorem \ref{ThB} (2),
though each $\overline{X}_n$ here has isolated singularities,
and the map $\overline{m}_n$ may not be \'etale.
$\overline{m}_n$ induces a cohomologically hyperbolic dominant rational map $\widetilde{m}_n \colon X_n \ratmap X_n$
which may not be holomorphic just like the similar construction on
smooth Kummer surfaces.
\end{example}
%
%
%
%
%
%

\end{large}

\begin{thebibliography}{99}
\bibitem{BHPV}
W. P. Barth, K. Hulek, C. A. M. Peters and A. Van de Ven,
Compact complex surfaces. Second edition.
Springer-Verlag, Berlin, 2004.
\bibitem{Bi}
G. Birkhoff,
Linear transformations with invariant cones,
Amer. Math. Monthly, {\bf 74} (1967), 274--276.
\bibitem{Cp}
F. Campana,
Connexit\'e rationnelle des vari\'et\'es de Fano,
Ann. Sci. \'ecole Norm. Sup. (4){\bf 25} (1992), no. 5, 539--545.
\bibitem{Ct}
S. Cantat,
Dynamique des automorphismes des surfaces $K3$.
Acta Math. {\bf 187} (2001), no. 1, 1--57.
\bibitem{Di}
T. -C. Dinh, Suites d'applications m¨¦romorphes multivalu¨¦es et
courants laminaires (French) [Sequences of multivalued meromorphic
maps and laminar currents],  J. Geom. Anal.  {\bf 15}  (2005) 207--227.
\bibitem{DS}
T. -C. Dinh and N. Sibony,
Regularization of currents and entropy.
Ann. Sci. \'ecole Norm. Sup. (4){\bf 37} (2004), no. 6, 959--971.
\bibitem{Fm}
Y. Fujimoto,
Endomorphisms of smooth projective 3-folds with non-negative Kodaira dimension.
Publ. Res. Inst. Math. Sci. {\bf 38}  (2002),  no. 1, 33--92.
\bibitem{FN}
Y. Fujimoto and N. Nakayama,
Endomorphisms of smooth projective 3-folds with nonnegative
Kodaira dimension, II, J. Math. Kyoto Univ. {\bf 47} (2007) 79–114.
\bibitem{GHS}
T. Graber, J. Harris and J. Starr,
Families of rationally connected varieties,
J. Amer. Math. Soc. {\bf 16} (2003) 57--67.
\bibitem{Gr}
M. Gromov,
Convex sets and K\"ahler manifolds. in:
Advances in differential geometry and topology, 1--38,
World Sci. Publ., Teaneck, NJ, 1990.
\bibitem{Gu05}
V. Guedj,
Ergodic properties of rational mappings with large topological degree.
Ann. of Math. (2) {\bf 161} (2005), no. 3, 1589--1607.
\bibitem{Gu06}
V. Guedj,
Propri\'et\'es ergodiques des applications rationnelles,
math.CV/{\bf 0611302}.
\bibitem{Ka}
Y. Kawamata,
Abundance theorem for minimal threefolds,
Invent. Math., {\bf 108} (1992) 229--246.
\bibitem{KMM}
Y. Kawamata, K. Matsuda and K. Matsuki, Introduction to the minimal
model problem. Algebraic geometry, Sendai, 1985, 283--360, Adv.
Stud. Pure Math., {\bf 10}, 1987.
\bibitem{KOZ}
J. Keum, K. Oguiso and D. -Q. Zhang,
Conjecture of Tits type for complex varieties and Theorem of Lie-Kolchin type for a cone,
Math. Res. Lett. (to appear); also: arXiv:math/{\bf 0703103}.
\bibitem{Ko}
J. Koll\'ar, Shafarevich maps and plurigenera of algebraic varieties.
Invent. Math.  {\bf 113}  (1993),  no. 1, 177--215.
\bibitem{KM}
J. Koll\'ar and S. Mori,
Birational geometry of algebraic varieties. Cambridge Tracts in Mathematics, {\bf 134}.
Cambridge University Press, Cambridge, 1998.
\bibitem{KoMM}
J. Koll\'ar, Y. Miyaoka and S. Mori,
Rational connectedness and boundedness of Fano manifolds.
J. Differential Geom. {\bf 36} (1992), no. 3, 765--779.
\bibitem{McP2}
C. McMullen, Dynamics on blowups of the projective plane.
Publ. Math. Inst. Hautes ¨¦tudes Sci. No.{\bf 105} (2007) 49--89.
\bibitem{Mi}
Y. Miyaoka,
Abundance conjecture for $3$-folds: case $\nu = 1$,
Comp. Math. {\bf 68} (1988) 203--220.
\bibitem{Ny}
N. Nakayama,
Compact K\"ahler manifolds whose universal covering spaces are
biholomorphic to $\BCC^n$,
(a modified version, but in preparation);
the original is RIMS preprint 1230,
Res. Inst. Math. Sci. Kyoto Univ. 1999.
\bibitem{NZ}
N. Nakayama and D. -Q. Zhang,
Building blocks of \'etale endomorphisms of complex projective manifolds.
RIMS preprint 1577,
Res. Inst. Math. Sci. Kyoto Univ. 2007.
http://www.kurims.kyoto-u.ac.jp/preprint/index.html
\bibitem{NS}
Yo. Namikawa and J. H. M. Steenbrink,
Global smoothing of Calabi-Yau threefolds,
Invent. Math. {\bf 122} (1995) 403--419.
\bibitem{Og06}
K. Oguiso,
Tits alternative in hyperk\"ahler manifolds, Math. Res.
Lett. {\bf 13} (2006) 307 -- 316.
\bibitem{Og08}
K. Oguiso,
Bimeromorphic automorphism groups of non-projective hyperk\"ahler manifolds -
a note inspired by C.T. McMullen, J. Differential Geom.
{\bf 78} (2008) 163--191.
\bibitem{OS}
K. Oguiso and J. Sakurai,
Calabi-Yau threefolds of quotient type.  Asian J. Math.  {\bf 5}  (2001),  no. 1, 43--77.
\bibitem{Ty}
S. Takayama,
Local simple connectedness of resolutions of log-terminal singularities,
Internat. J. Math. {\bf 14} (2003), no. 8, 825--836.
\bibitem{Un}
K. Ueno,
Classification theory of algebraic varieties and compact complex spaces.
Lecture Notes in Mathematics, Vol. {\bf 439}.
Springer-Verlag, Berlin-New York, 1975.
\bibitem{Zh1}
D. -Q. Zhang,
Automorphism groups and anti-pluricanonical curves,
Math. Res. Lett. {\bf 15} (2008), No. 1, 163 - 183.
\bibitem{Zh2}
D. -Q. Zhang,
Dynamics of automorphisms on projective complex manifolds,
preprint 2006; also Max Planck Institute for Mathematics, MPIM{\bf 2007-19} at:
\newline
www.mpim-bonn.mpg.de/Research/MPIM+Preprint+Series/
\bibitem{ICCM}
D. -Q. Zhang,
Dynamics of automorphisms of compact complex manifolds,
Proceedings of The Fourth International Congress of Chinese Mathematicians (ICCM2007),
Vol II, pp. 678 - 689; also: arXiv:{\bf 0801.0843.}

\end{thebibliography}
\end{document}